\documentclass[12pt]{amsart}
\usepackage[T1]{fontenc}
\usepackage{lmodern}
\usepackage[margin=1in]{geometry}
\usepackage{amsmath,amssymb,amsthm,mathtools,microtype}
\usepackage{xurl,needspace}
\usepackage{xcolor}
\usepackage[hidelinks]{hyperref}
\hypersetup{pdftitle={On a conjecture of Browning and Sawin on random hypersurfaces with sign coefficients},pdfauthor={Ken Ono and Ashvin Swaminathan}}
\numberwithin{equation}{section}
\newtheorem{theorem}{Theorem}[section]
\newtheorem{proposition}[theorem]{Proposition}
\newtheorem{lemma}[theorem]{Lemma}
\newtheorem{corollary}[theorem]{Corollary}
\newtheorem*{conjecture}{Conjecture}
\theoremstyle{definition}
\newtheorem*{remark}{Remark}
\newtheorem{example}[theorem]{Example}
\newcommand{\F}{\mathbb F}
\newcommand{\Z}{\mathbb Z}
\newcommand{\Q}{\mathbb Q}
\newcommand{\C}{\mathbb C}
\newcommand{\A}{\mathbb A}
\newcommand{\PP}{\mathbb P}
\newcommand{\Prob}{\mathbf P}
\newcommand{\Sing}{\operatorname{Sing}}
\newcommand{\eps}{\varepsilon}
\newcommand{\m}{\mathfrak m}
\newcommand{\del}{\partial}
\newcommand{\Th}{\Theta}

\title[On a conjecture of Browning and Sawin]{On a conjecture of Browning and Sawin on random hypersurfaces with sign coefficients}
\author{Ken Ono}
\address{Axiom Math, 124 University Avenue, Palo Alto, CA 94301, USA}
\email{ken@axiommath.ai}
\author{Ashvin A.~Swaminathan}
\address{Axiom Math, 124 University Avenue, Palo Alto, CA 94301, USA}
\email{ashvin@axiommath.ai}
\date{}
\begin{document}
\begin{abstract}
Browning and Sawin conjectured that random hypersurfaces with sign coefficients are smooth with
probability tending to one as the degree grows. We prove this conjecture and obtain a quantitative
bound. For each $n\geq1$, a degree $d$ form in $n+1$ variables, with independent uniform
coefficients in $\{-1,1\}$, defines a singular complex hypersurface with probability
$O_n(d^{-1/2})$. The positive-dimensional singular loci occur with exponentially small
probability. For $n\geq3$, the same exponential bound holds for failure of absolute
irreducibility. These results have been formalized in Lean by AxiomProver assuming existing literature.
\end{abstract}
\maketitle

\section{Introduction}\label{sec:intro}

A general homogeneous polynomial over $\C$ defines a smooth hypersurface. It is natural to ask
whether smoothness remains likely when every coefficient is restricted to a sign $\pm 1$. Generic
smoothness alone does not answer this question. As the degree increases, both the coefficient
space and the exceptional algebraic set of singular forms change, and the geometric statement
does not control the proportion of sign vectors in that set. Browning and Sawin formulated this
question as a conjecture in their work on random Diophantine equations of large degree
\cite[Conjecture~1.3]{BS}. Here we prove their conjecture and obtain a quantitative bound.

Fix integers $n\geq1$ and $d\geq3$. For a tuple 
$\alpha:=(\alpha_0,\dots, \alpha_n)$ with nonnegative integer entries, we let  $$|\alpha|\coloneqq\alpha_0+\cdots+\alpha_n\qquad  {\text {\rm and}}\qquad x^\alpha\coloneqq x_0^{\alpha_0}\cdots
x_n^{\alpha_n}.
$$
Let $\Prob$ denote probability, and choose independent random signs
$\eps_\alpha$, each equal to $1$ or $-1$ with probability $1/2$. We study
\begin{equation}\label{eq:model}
 f(x_0,\ldots,x_n)\coloneqq\sum_{|\alpha|=d}\eps_\alpha x^\alpha.
\end{equation}
Every degree $d$ monomial occurs in this sum. Let $\mathcal B_{d,n}$ denote the set of these
forms, and put $N_{d,n}\coloneqq\binom{n+d}{n}$. Therefore, \eqref{eq:model} gives the uniform
distribution on the $2^{N_{d,n}}$ elements of $\mathcal B_{d,n}$.

Write $\PP^n_k$ for projective $n$ space over a field $k$, and let $X_f\subset\PP^n_\C$ be the
closed hypersurface defined by $f$. We retain its scheme structure, including any repeated
factors. By smoothness we mean geometric smoothness, which excludes singular points over an
algebraic closure, not just rational or real singular points. 

In this notation, Browning and Sawin formulated the following conjecture.

\begin{conjecture}[Browning--Sawin {\cite[Conjecture~1.3]{BS}}]
For every integer $n\geq2$, we have
\[
 \lim_{d\to\infty}
 \frac{\#\{f\in\mathcal B_{d,n}:X_f\text{ is smooth}\}}
      {2^{N_{d,n}}}=1.
\]
\end{conjecture}

\noindent
Browning and Sawin identify the defining forms $f$ and $-f$, whereas we count both forms.
This does not change the smoothness probability. We prove the following quantitative result,
which also includes projective dimension one.

\begin{theorem}\label{thm:main}
For every integer $n\geq1$, there is a constant $C_n>0$ such that the following bound
holds for every integer $d\geq3$:
\begin{equation}\label{eq:main}
 \Prob(X_f\text{ is singular})\leq C_n d^{-1/2}.
\end{equation}
\end{theorem}

The following corollary immediately confirms the conjecture.

\begin{corollary}[The Browning--Sawin conjecture]\label{cor:BS}
The conjecture above holds. More precisely, for every fixed integer $n\geq2$, the proportion
of smooth forms satisfies
\[
 \frac{\#\{f\in\mathcal B_{d,n}:X_f\text{ is smooth}\}}
      {2^{N_{d,n}}}=1+O_n(d^{-1/2})
 \qquad(d\to\infty).
\]
\end{corollary}

\begin{remark}
The assumption that the coefficients are nonzero is important. Suppose instead that each
coefficient is chosen independently from $0$ and $1$, with equal probability. If the
coefficients of $x_0^d$ and $x_0^{d-1}x_i$, for $1\leq i\leq n$, are all zero, then the
polynomial and all its affine first derivatives vanish at $[1:0:\cdots:0]$. Requiring the
coefficient of $x_1^d$ to be $1$ also ensures that the polynomial is nonzero. For $d\geq2$,
these are $n+2$ distinct coefficient conditions. The event that they all hold has probability
$2^{-(n+2)}$, independently of $d$, and gives a singular hypersurface. Therefore, allowing zero
coefficients can prevent the smoothness probability from tending to one. 
\end{remark}

\begin{remark}
In dimension one, the case of Littlewood polynomials (i.e. the coefficients are in $\{ \pm 1\}$),  smoothness means no
repeated roots. Peled, Sen, and
Zeitouni \cite[Theorems~1.1 and~1.2]{PSZ} give the sharper bound $O(d^{-2})$ in this case.
\end{remark}

Our second result gives a stronger bound for positive-dimensional singular loci than for
isolated singular points. Let $\Sing X_f$ denote the closed set of singular points, with
the convention $\dim\varnothing\coloneqq-1$.

\begin{theorem}\label{thm:positive}
For all integers $n\geq1$ and $d\geq3,$ we have
\begin{equation}\label{eq:positive}
 \Prob(\dim\Sing X_f\geq1)
 \leq \frac{n(n+1)d^{n-1}}{2^{\lfloor(d-1)/3\rfloor+1}}.
\end{equation}
\end{theorem}

A polynomial is \emph{absolutely irreducible} if it is irreducible over an algebraic closure of
its coefficient field. For our integral forms, this is equivalent to irreducibility over
$\C$. In dimension at least three, a reducible hypersurface has a positive-dimensional
singular locus. Therefore, the preceding theorem gives the following corollary.

\begin{corollary}\label{cor:absolute}
For $n\geq3$, we have that
$$
\Prob( f \ {\text {is not absolutely irreducible}})\leq 
\frac{n(n+1)d^{n-1}}{2^{\lfloor(d-1)/3\rfloor+1}}.
$$
\end{corollary}

\noindent
Indeed, two hypersurface factors in $\PP^n$ meet in dimension at least $n-2$, and every point
of their intersection is singular for the product. Section~\ref{sec:positive} gives the
argument, including the case of repeated factors.

\begin{remark}
For coefficients uniformly distributed over a fixed finite field, Poonen proves the
corresponding density-zero statement for positive-dimensional singular loci and deduces
irreducibility in the same way (see \cite[Theorem~3.2]{Poo} and the remark following it).
Our estimates apply to sign coefficients and are uniform in the prime by which we reduce.
\end{remark}

\begin{remark}
Irreducibility over $\Q$ is already known for this model with probability tending to one.
Browning and Sawin observe (see the discussion preceding \cite[Conjecture~1.3]{BS}) that the
work of Bary-Soroker and Kozma \cite{BSK} gives asymptotic irreducibility of ternary sign
forms over $\Q$, and restricting to three variables extends this to every fixed $n\geq2$.
Corollary~\ref{cor:absolute} concerns irreducibility over $\C$, with an exponential bound, and
is obtained from the singular locus rather than from an irreducibility theorem for random
polynomials. In a related direction, Kozma and Zeitouni \cite[Theorem~1]{KZ} study common
zeros of independent Bernoulli polynomials.
\end{remark}

\subsection*{Strategy of proof}
Our starting point is Poonen's proof of the Bertini theorem over finite fields
\cite[Section~2]{Poo}. He groups the possible singular points by their residue-field degree,
treats points of small degree by a local estimate, and treats points of large degree by a
sieve that uses Frobenius powers to impose the derivative equations one at a time. Both steps
rely on the coefficients being uniformly distributed over the finite field, a property that is
preserved when random polynomials of controlled shape are added to $f$. The sign distribution
has no such invariance. We therefore work only with the coefficients in
\eqref{eq:model}. We partition them in a way that supplies the independence that Poonen obtains from
randomization, and then we apply a concentration estimate for sign sums in place of local uniformity.

The reduction to a finite field is Lemma~\ref{lem:specialization},  a sign form that is
singular over $\C$ remains singular modulo every prime, and the dimension of its singular
locus cannot decrease. Fix an odd prime $p$. On a standard affine chart, the dehomogenized
form carries an independent sign on every monomial of total degree at most $d$, and we
estimate its singularity probability separately for closed points of small and of large
degree.

For a closed point $P$ of degree $e$, singularity at $P$ is the vanishing of the image of
the polynomial in the first-order algebra $$\F_p[y_1,\ldots,y_n]/\m_P^2,
$$
an $\F_p$-vector
space of dimension $(n+1)e$. Away from the coordinate origin some coordinate is a unit in
this algebra, and multiplying a monomial basis by its powers yields many bases with pairwise
disjoint monomial supports. A Fourier--H\"older argument of He, Pham, and Xu
(see the proof of Proposition 3.3 of \cite{HPX}) stated for arbitrary disjoint bases in
Lemma~\ref{lem:fourier}, then bounds the probability of any prescribed image by a quantity
that decays with the number of bases. Summing over the closed points of $\PP^n_{\F_p}$ of
degree at most a cutoff $R$ gives the low-degree contribution.

For points of degree exceeding $R$, there are too many candidates for this union bound.
Sorting the exponent vectors by their residues modulo $p$ decomposes the dehomogenized form
as 
$$U+\sum_{i=1}^n y_iG_i^p+H^p,
$$
where $U,G_1,\ldots,G_n,H$ use disjoint sets of the original
signs and are therefore independent. The $i$th partial derivative involves only $U$ and
$G_i$, and $H$ appears in no derivative. Revealing $G_1,\ldots,G_n$ in turn and applying an
affine B\'ezout bound, we show that with high probability each derivative equation lowers
the dimension of the critical locus until it is finite. The block $H$ then decides whether a
critical point lies on the hypersurface, and a point of large degree provides many
independent evaluation directions for $H$, which makes this probability exponentially small.

Theorem~\ref{thm:finite} combines the two estimates for an arbitrary odd prime and cutoff.
Taking $p$ of order $\sqrt d$ and $R$ of order $\log d$, the points of degree one contribute
$O_n(d^{-1/2})$ and every other term is smaller. This proves Theorem~\ref{thm:main}. A
positive-dimensional singular locus forces a positive-dimensional critical locus on some
chart, so Theorem~\ref{thm:positive} needs only the high-degree sieve, applied at the fixed
prime $p=3$.

The paper is organized as follows. Section~\ref{sec:specialization} carries out the
reduction modulo a prime. Section~\ref{sec:concentration} proves the concentration estimates
for sign sums, and Sections~\ref{sec:local} and~\ref{sec:high} treat closed points of small
and large degree. Section~\ref{sec:finite} combines these into a finite-field estimate, and
Section~\ref{sec:proofs} deduces the main results.

\section*{Acknowledgements} \noindent The authors thank Tim Browning for bringing this problem to their attention, as well as his enthusiastic support for this project.

\section{Singularities and reduction modulo a prime}\label{sec:specialization}

To reduce the problem to a finite field, we use the equation of the hypersurface and its
first derivatives to describe the singular locus. We then show that the dimension of a
projective singular locus cannot decrease under reduction.

To make this all precise, we begin by setting notation. For a field $k$, write $\bar k$ for an algebraic closure and $\A^n_k$ for affine
$n$-space. Given polynomials $h_1,\ldots,h_j$ over $k$, define their common geometric zero set by
\[
 Z(h_1,\ldots,h_j)\coloneqq
 \{a\in\A^n_{\bar k}:h_1(a)=\cdots=h_j(a)=0\}.
\] All dimensions and irreducible components below are
geometric. The singular locus is considered with its reduced structure whenever only its
dimension is used.

\subsection{The affine singularity criterion}

We first recall how the linear part of a local equation determines smoothness. We then pass
to homogeneous equations, including in characteristics that divide the degree.

For a polynomial $g$ in affine coordinates $y_1,\ldots,y_n$, write $\del_i g\coloneqq\partial g/\partial
y_i$. Suppose that $g$ is nonzero, and let $a$ be a point of $\A^n_{\bar k}$ at which $g$ vanishes.
We introduce local coordinates
$z\coloneqq(z_1,\ldots,z_n)$ centered at $a$. The first-order expansion is
\[
 g(a+z)\equiv \sum_{i=1}^n(\del_i g)(a)z_i
       \pmod{(z_1,\ldots,z_n)^2}.
\]
The local hypersurface has dimension $n-1$. The displayed linear form cuts out its tangent
space, which has dimension $n-1$ when the form is nonzero and dimension $n$ otherwise.
Therefore, the regular-local-ring criterion gives
\begin{equation}\label{eq:jacobian}
 a\text{ is singular on }g=0
 \quad\Longleftrightarrow\quad
 g(a)=\del_1g(a)=\cdots=\del_ng(a)=0.
\end{equation}
Over a perfect field, geometric regularity agrees with smoothness. This is the hypersurface
Jacobian criterion (see \cite[Chapter~I, \S5; Chapter~II, \S8]{Hartshorne}).

For a homogeneous polynomial $F$ of positive degree $d$, use $\del_iF\coloneqq\partial F/\partial x_i$
in homogeneous coordinates. On the chart $x_0\ne0$, put $y_i\coloneqq x_i/x_0$ and $g(y)\coloneqq F(1,y)$. Direct
differentiation and homogeneity give
\begin{equation}\label{eq:homogeneous-derivatives}
 \del_iF(1,y)=\del_i g(y)\quad(i\geq1),\qquad
 \del_0F(1,y)=dg(y)-\sum_{i=1}^n y_i\del_i g(y).
\end{equation}
Therefore, the homogeneous equations $F=\del_0F=\cdots=\del_nF=0$ define the projective singular locus
in every characteristic. In particular, the equation $F=0$ is retained when the characteristic
divides $d$.

\subsection{Specialization of projective equations}
We need singularities to persist after reduction modulo a prime. Projectivity gives this
implication and preserves a lower bound for the dimension of the singular locus. No lifting
from a finite field is required.

An integral polynomial is \emph{primitive} if its coefficients have greatest common divisor
$1$. Such a polynomial has nonzero reduction modulo every prime. The following lemma is a
special case of upper semicontinuity of fiber dimension for proper morphisms
\cite[Tag~0D4I]{Stacks}. For completeness, we give a proof using graded coordinate rings.

\begin{lemma}\label{lem:specialization}
Let $I\subset\Z[x_0,\ldots,x_n]$ be a homogeneous ideal. For a field $k$, let $S_k$ be the
projective scheme defined by its image in $k[x_0,\ldots,x_n]$. For every prime $p$, the fiber
dimensions satisfy
\begin{equation}\label{eq:dimension-specialization}
 \dim S_{\overline{\F}_p}\geq\dim S_\C.
\end{equation}
Consequently, if $F$ is a primitive integral homogeneous polynomial of positive degree and $\bar
F$ is its reduction modulo $p$, then the singular loci satisfy
\begin{equation}\label{eq:sing-specialization}
 \dim\Sing X_{\bar F}\geq\dim\Sing X_F,
\end{equation}
where the two hypersurfaces are taken over $\overline{\F}_p$ and $\C$, respectively.
\end{lemma}
\begin{proof}
Put $B\coloneqq\Z[x_0,\ldots,x_n]/I$, and write $B_m$ for its homogeneous part of degree $m$. Each $B_m$
is a finitely generated abelian group. Its dimension after tensoring with $\Q$ is its rank,
whereas its dimension after tensoring with $\F_p$ is at least that rank. Hence
\[
 \dim_{\F_p}(B_m\otimes\F_p)
 \geq\dim_\Q(B_m\otimes\Q)\qquad(m\geq0).
\]
These dimensions are the Hilbert functions of the graded coordinate rings. For sufficiently
large $m$, each agrees with its Hilbert polynomial. A nonempty projective scheme of dimension
$r$ has a Hilbert polynomial of degree $r$ with positive leading coefficient. The Hilbert
polynomial of an empty projective scheme is zero. The inequality therefore shows that the
special fiber cannot have smaller dimension than the generic fiber. Extending the ground
field preserves both the Hilbert function and dimension, which proves
\eqref{eq:dimension-specialization}. These standard facts about Hilbert polynomials are
recalled in \cite[Chapter~I, \S7]{Hartshorne}.

Apply the first assertion to $I\coloneqq(F,\del_0F,\ldots,\del_nF)$. Primitivity ensures that each
reduction remains a nonzero hypersurface equation. Equations~\eqref{eq:jacobian}
and~\eqref{eq:homogeneous-derivatives} identify the support of each projective zero scheme with
the corresponding singular locus. This proves \eqref{eq:sing-specialization}.
\end{proof}

Every sign form is primitive. From now on, for a prime $p$, write $X_{\bar
f}\subset\PP^n_{\overline{\F}_p}$ for the reduced-coefficient hypersurface, without discarding
its scheme structure. Lemma~\ref{lem:specialization} gives
\begin{equation}\label{eq:prob-specialization}
 \Prob(X_f\text{ is singular})
 \leq\Prob(X_{\bar f}\text{ is singular}).
\end{equation}
{\color{black}
The converse need not hold, since a smooth complex hypersurface can become singular after
reduction. We do not use that implication.
}
\section{Concentration of sign sums}\label{sec:concentration}

We require two concentration estimates. The first uses linear independence to give a bound
that decreases exponentially with the number of independent directions. The second applies
to a finite target space and uses disjoint bases to obtain a stronger bound that also depends
on the size of the field.

For the first estimate, fix the coefficients outside a linearly independent set of columns.
A prescribed sum then determines at most one assignment of signs to those columns.

\begin{lemma}\label{lem:rank}
Let $k$ be a field of characteristic different from $2$, and let $v_1,\ldots,v_M$ belong to a
$k$-vector space. Suppose that $r$ of these vectors are linearly independent over $k$. For
independent uniform signs $\eps_1,\ldots,\eps_M$ and any vector $u$, we have
\[
 \Prob\left(\sum_{j=1}^M\eps_jv_j=u\right)\leq2^{-r}.
\]
\end{lemma}
\begin{remark}
The conclusion also applies when the ambient space is viewed as a vector space over a subfield
of $k$ and the chosen vectors are independent over that subfield.
\end{remark}

\begin{proof}[Proof of Lemma~\ref{lem:rank}]
Condition on every coefficient except those attached to the chosen independent vectors. At most
one assignment of the remaining signs can give the prescribed sum. Indeed, two different
assignments would give a nontrivial linear relation with coefficients in $\{0,2,-2\}$. Each
assignment has probability $2^{-r}$.
\end{proof}

For an odd prime $p$ and an integer $K\geq1$, define the average below, where each residue
$a\in\F_p$ may be represented by any integer:

\begin{equation}\label{eq:theta-def}
 \Th_K(p)\coloneqq\frac1p\sum_{a\in\F_p}\left|\cos\frac{2\pi a}{p}\right|^K.
\end{equation}
For fixed $p$, this average is nonincreasing in $K$, since every absolute cosine lies in $[0,1]$.

The following lemma gives the Fourier--H\"older argument of He, Pham, and Xu
\cite[proof of Proposition~3.3]{HPX} for arbitrary disjoint bases. Their bases come from
consecutive univariate monomials. We use uniform signs and allow any prescribed image, as
needed for the first-order algebras in Section~\ref{sec:local}. Independence is required only
for the original coefficients.

\begin{lemma}\label{lem:fourier}
Let $L,M,K$ be positive integers, and let $v_1,\ldots,v_M$ be vectors in $\F_p^L$.
Let $\eps_1,\ldots,\eps_M$ be
independent uniform signs. Suppose that the vectors contain $K$ pairwise disjoint groups of
indices, each giving a basis of $\F_p^L$. Then the largest point probability satisfies
\begin{equation}\label{eq:fourier}
 \sup_{u\in\F_p^L}\Prob\left(\sum_{j=1}^M\eps_jv_j=u\right)
 \leq\Th_K(p)^L.
\end{equation}
\end{lemma}
\begin{proof}
For $\xi,v\in\F_p^L$, let $\langle\xi,v\rangle\coloneqq\sum_{i=1}^L\xi_i v_i$. The corresponding
additive character is $v\mapsto\exp(2\pi i\langle\xi,v\rangle/p)$, where a residue is
represented by any integer. Character orthogonality says that, for $w\in\F_p^L$,
\[
 p^{-L}\sum_{\xi\in\F_p^L}\exp(2\pi i\langle\xi,w\rangle/p)
 =\begin{cases}1,&w=0,\\0,&w\ne0.\end{cases}
\]
For a nonzero vector $w$, this follows by summing first in a coordinate where $w$ is nonzero. Apply the
identity to $w\coloneqq\sum_j\eps_jv_j-u$ and take expectations. The character mean on $\eps_jv_j$ is
$\cos(2\pi\langle\xi,v_j\rangle/p)$. Independence and the triangle inequality give
\[
 \Prob\left(\sum_j\eps_jv_j=u\right)
 \leq p^{-L}\sum_{\xi\in\F_p^L}
       \prod_{j=1}^M\left|\cos\frac{2\pi\langle\xi,v_j\rangle}{p}\right|.
\]
Discard the factors outside the selected bases, since they are at most $1$. Write the $h$th
basis as $b_{h1},\ldots,b_{hL}$, and set
\[
 F_h(\xi)\coloneqq\prod_{\ell=1}^L
       \left|\cos\frac{2\pi\langle\xi,b_{h\ell}\rangle}{p}\right|.
\]
H\"older's inequality, applied to the uniform probability measure on $\F_p^L$, yields
\[
 p^{-L}\sum_\xi\prod_{h=1}^K F_h(\xi)
 \leq\prod_{h=1}^K\left(p^{-L}\sum_\xi F_h(\xi)^K\right)^{1/K}.
\]
For every $h$, the map $\xi\mapsto(\langle\xi,b_{h\ell}\rangle)_{\ell=1}^L$ is a bijection of
$\F_p^L$. Consequently,
\[
 p^{-L}\sum_\xi F_h(\xi)^K
 =\left(\frac1p\sum_{a\in\F_p}
       \left|\cos\frac{2\pi a}{p}\right|^K\right)^L
 =\Th_K(p)^L.
\]
Substitution into H\"older's inequality proves the result. The coordinates of a random image
in $\F_p^L$ need not be independent. The argument uses only the disjointness of the original
coefficients belonging to the selected bases.
\end{proof}

To apply the lemma, we need the following bound in terms of the prime and the number of bases.

\begin{lemma}\label{lem:theta}
For every odd prime $p$ and integer $K\geq1$, the cosine average satisfies
\begin{equation}\label{eq:theta}
 \Th_K(p)\leq\frac1p+\sqrt{\frac{2}{\pi K}}
 \leq\frac1p+\frac1{\sqrt K}.
\end{equation}
In particular, $\Th_{p^2}(p)$ is less than $2/p$.
\end{lemma}
\begin{proof}
Multiplication by $2$ permutes the residues modulo $p$. Taking absolute values removes changes
of the angle by multiples of $\pi$. Consequently
\[
 p\Th_K(p)=1+2\sum_{r=1}^{(p-1)/2}\cos(\pi r/p)^K.
\]
On the interval $0\leq x<\pi/2$, integration of the inequality $\tan x\geq x$ gives
$\log\cos x\leq-x^2/2$. Hence, we conclude that
\[
 \Th_K(p)\leq\frac1p+
 \frac2p\sum_{r\geq1}\exp\left(-\frac{\pi^2Kr^2}{2p^2}\right).
\]
The summand is decreasing as a function of $r\geq0$. Bounding the sum by the integral from $0$
to infinity gives
\[
 \frac2p\sum_{r\geq1}\exp\left(-\frac{\pi^2Kr^2}{2p^2}\right)
 \leq\frac2p\int_0^\infty
       \exp\left(-\frac{\pi^2Kx^2}{2p^2}\right)\,dx
 =\sqrt{\frac2{\pi K}}.
\]
This proves \eqref{eq:theta}. The final assertion follows because $\sqrt{2/\pi}$ is less than $1$.
\end{proof}

\section{Singularity at a prescribed closed point}\label{sec:local}

To use the concentration estimates, we express singularity at a prescribed point as the
vanishing of a vector in a finite algebra. We keep the residue-field degree explicit so that
we can later sum the resulting probabilities over projective points.
To this end,
fix an odd prime $p$, write $\F_{p^e}$ for the degree-$e$ extension of $\F_p$, and put
$S\coloneqq\F_p[y_1,\ldots,y_n]$. A closed point $P$ of $\A^n_{\F_p}$ is
specified by a maximal ideal $\m_P\subset S$. Define its residue field and degree by
\[
 \kappa(P)\coloneqq S/\m_P,\qquad
 \deg P\coloneqq[\kappa(P):\F_p].
\]
Equivalently, the closed point is an orbit of this size under the map that raises each
geometric coordinate to the $p$th power. For $g\in S$, the notation $g(P)$
means its image in $\kappa(P)$.

The singularity criterion requires both $g(P)$ and its first derivatives. Following Poonen's
local sieve \cite[Lemma~2.3]{Poo}, we use the finite algebra describing the first-order
neighborhood of $P$:
\[
 A_P\coloneqq S/\m_P^2.
\]

The following lemma shows that the algebra records the value and first derivatives at the
point and at each of its conjugates. For a nonzero hypersurface equation, singularity is
therefore the vanishing of a single vector over $\F_p$.

\begin{lemma}\label{lem:jet}
Put $e\coloneqq\deg P$. The algebra $A_P$ has dimension $(n+1)e$ over $\F_p$. A polynomial $g\in S$
maps to zero in $A_P$ if and only if
\begin{equation}\label{eq:sing-affine}
 g(P)=\del_1g(P)=\cdots=\del_ng(P)=0.
\end{equation}
If $g$ is nonzero, this is equivalent to the hypersurface $g=0$ being singular at $P$.
\end{lemma}
\begin{proof}
Let $\bar k\coloneqq\overline{\F}_p$. Since $\kappa(P)/\F_p$ is separable, the ideal $\m_P\bar k[y]$ is
the intersection of the pairwise comaximal maximal ideals of the $e$ conjugate geometric points
$P_1,\ldots,P_e$. For comaximal ideals, intersection equals product. The same assertion applied
to their squares, followed by the Chinese remainder theorem, gives
\[
 A_P\otimes_{\F_p}\bar k
 \simeq\prod_{j=1}^e\bar k[y_1,\ldots,y_n]/\m_{P_j}^2.
\]
Write $a_1,\ldots,a_n$ for the coordinates of $P_j$. The corresponding factor has basis
\[
 1,\ y_1-a_1,\ldots,y_n-a_n.
\]
Its dimension is $n+1$. The first-order Taylor expansion is
\[
 g\equiv g(P_j)+\sum_{i=1}^n\del_i g(P_j)(y_i-a_i)
       \pmod{\m_{P_j}^2}.
\]
This proves the dimension formula. It also proves \eqref{eq:sing-affine}: extending scalars to a
field does not turn a nonzero vector into zero, and the equations are preserved by conjugation.
For nonzero $g$, the equivalence with singularity is the criterion \eqref{eq:jacobian}.
\end{proof}

\subsection{Finding disjoint monomial bases}

To apply the concentration bound, we construct monomial bases with disjoint coefficient
supports. We first find one basis of bounded degree and then multiply it by powers of a
coordinate that is invertible at the point.

We use the elementary filtration argument from Poonen's interpolation and evaluation
estimates \cite[Lemmas~2.1 and~2.5]{Poo}. We state it for a finite algebra so that it applies
to both residue fields and first-order neighborhoods. The observation is that the degree
filtration increases in dimension until it is the whole algebra.

\begin{lemma}\label{lem:filtration}
Let $A$ be a nonzero finite-dimensional commutative unital algebra over a field $k$. Suppose
that $A$ is generated as a $k$-algebra by $y_1,\ldots,y_n$, and put $L\coloneqq\dim_k A$.
For an integer $s\geq0$, define the degree filtration by
\[
 V_s\coloneqq\operatorname{span}_k
 \{y_1^{a_1}\cdots y_n^{a_n}:a_i\in\Z_{\geq0},\ a_1+\cdots+a_n\leq s\}.
\]
Then its dimension satisfies
\begin{equation}\label{eq:filtration}
 \dim_k V_s\geq\min(s+1,L).
\end{equation}
In particular, $A$ has a basis consisting of images of monomials of degree at most $L-1$.
\end{lemma}
\begin{proof}
We have $V_0=k\cdot1$. If $V_s$ and $V_{s+1}$ coincide, then $V_s$ is stable under multiplication by every
generator. Since it contains $1$, it is all of $A$. Otherwise the dimension increases by at
least one. This proves \eqref{eq:filtration}. A basis may be selected from any spanning set.
\end{proof}

At a point other than the origin, some coordinate is a unit in the first-order algebra.
Multiplying by sufficiently separated powers of this coordinate gives bases represented by
disjoint sets of monomials.

\begin{lemma}\label{lem:packing}
Let $P$ be a closed point other than the origin, and put $L\coloneqq(n+1)\deg P$. Suppose that the
positive integer $K$ satisfies $KL\leq d+1$. The monomials of
total degree at most $d$ contain $K$ pairwise disjoint groups whose images form bases of $A_P$.
\end{lemma}
\begin{proof}
Choose a coordinate nonzero at $P$, and relabel it $y_1$. The algebra $A_P$ is local, with
unique maximal ideal $\m_P/\m_P^2$. Since $y_1$ has nonzero image in $\kappa(P)$, it is a unit
in $A_P$. By Lemma~\ref{lem:filtration}, there is a monomial basis $B$ of degree at most $L-1$.
Use
\[
 B_h\coloneqq y_1^{hL}B,\qquad 0\leq h<K.
\]
Multiplication by a unit preserves a basis. Each monomial of $B$ has $y_1$-exponent between $0$
and $L-1$. The $y_1$-exponents in $B_h$ therefore lie between $hL$ and $(h+1)L-1$, so different
groups contain no common formal monomial. Every monomial in $B_h$ has total degree at most
$(h+1)L-1\leq d$.
\end{proof}

The following example gives two disjoint groups of monomials that map to bases of the
three-dimensional algebra at a rational point of an affine plane. Their combined images
need not be linearly independent.

\begin{example}\label{ex:jet}
Choose $a,b\in\F_p$ with $a$ nonzero, and consider the affine-plane point
$P\coloneqq(a,b)$. In $A_P$, put $u\coloneqq y_1-a$ and $v\coloneqq y_2-b$. These elements give the presentation
\[
 A_P\simeq\F_p[u,v]/(u,v)^2,
 \qquad y_1=a+u,\quad y_2=b+v.
\]
The monomials $1,y_1,y_2$ form a basis. Multiplication by the unit $y_1^3$ gives another basis,
represented by $y_1^3,y_1^4,y_1^3y_2$. For $d\geq4$, these two groups use six distinct
coefficients of a polynomial containing every monomial of total degree at most $d$. Their images
need not be independent of one another; each group need only be a basis. This is exactly the
hypothesis of Lemma~\ref{lem:fourier}.
\end{example}

\subsection{The local probability and the projective count}

We can now bound the singularity probability at a prescribed point. The bound depends on
its residue degree and can therefore be summed over projective points of bounded degree.

\begin{proposition}\label{prop:local}
For a tuple $\beta$ with nonnegative integer entries $\beta_1,\ldots,\beta_n$, write
$|\beta|\coloneqq\sum_i\beta_i$ and $y^\beta\coloneqq\prod_i y_i^{\beta_i}$. Let
\[
 g\coloneqq\sum_{|\beta|\leq d}\eps_\beta y^\beta\in\F_p[y_1,\ldots,y_n],
\]
where the coefficients are independent uniform signs.
Let $P$ be a closed point of degree $e$ other than the origin, and let the positive integer $K$ satisfy
$(n+1)eK\leq d+1$. Then the singularity probability at $P$ satisfies
\begin{equation}\label{eq:local}
 \Prob(g=0\text{ is singular at }P)
 \leq\Th_K(p)^{(n+1)e}.
\end{equation}
\end{proposition}
\begin{proof}
Apply Lemma~\ref{lem:fourier} to the images of the monomials in the $\F_p$-vector space $A_P$,
using the bases in Lemma~\ref{lem:packing}. By Lemma~\ref{lem:jet}, singularity is the event
that the resulting random sum is zero.
\end{proof}

Let $\bar f$ be the reduction of the original homogeneous sign form modulo $p$. The $n+1$
coordinate vertices have only one nonzero homogeneous coordinate. None belongs to
$X_{\bar f}$, since the value of the form there is a nonzero pure-power coefficient. Every
other projective closed point lies in a standard affine chart where it is not the origin.
Dehomogenization gives one independent sign on each monomial of total degree at most $d$,
without repetitions, and does not change the residue field. Proposition~\ref{prop:local}
therefore applies at every possible singular point. We choose one chart for each point,
so the projective count does not require an additional factor of $n+1$.

For each positive integer $e$, define
\[
 C_e\coloneqq\#\{P:P\text{ is a closed point of }\PP^n_{\F_p},\ \deg P=e\}.
\]
Each of these closed points contributes $e$ distinct $\F_{p^e}$-points. Therefore
\begin{equation}\label{eq:point-count}
 eC_e\leq\#\PP^n(\F_{p^e})
 =1+p^e+\cdots+p^{ne}
 \leq\frac{p}{p-1}p^{ne}.
\end{equation}
For an integer $R$ with $0\leq R\leq\lfloor(d+1)/(n+1)\rfloor$, set
\begin{equation}\label{eq:Ke}
 K_e\coloneqq\left\lfloor\frac{d+1}{(n+1)e}\right\rfloor
 \quad(1\leq e\leq R).
\end{equation}
Each $K_e$ is positive. A union bound gives
\begin{equation}\label{eq:low}
 \Prob(\exists\text{ singular closed }P:\deg P\leq R)
 \leq\frac{p}{p-1}\sum_{e=1}^R
       \frac{p^{ne}}e\Th_{K_e}(p)^{(n+1)e}.
\end{equation}
The empty sum is zero. This union bound requires no independence between different points
or affine charts.

\section{Closed points of large degree}\label{sec:high}

For points of large degree, we adapt Poonen's sieve \cite[Lemma~2.6]{Poo}. We partition the
coefficients rather than use the randomization available for uniform finite-field
coefficients.
The local estimate concerns a point fixed in advance. Here we first use some coefficients
to make the critical locus finite with high probability. We then use the remaining
coefficients to estimate the probability that a critical point lies on the hypersurface.
The partition below keeps these choices independent.

\subsection{Partitioning the coefficients}

Frobenius allows us to give each derivative its own random coefficient block and reserve
another block for the hypersurface equation. We do this by partitioning the existing
coefficients, so their sign distribution is unchanged.

Work on one standard affine chart. In this section, the indices $\alpha$ and $\beta$ are tuples
in $\Z_{\geq0}^n$, with the affine monomial notation introduced in Section~\ref{sec:local}.
Write
\[
 g(y)\coloneqq\sum_{|\alpha|\leq d}\eps_\alpha y^\alpha,
 \qquad s\coloneqq\left\lfloor\frac{d-1}{p}\right\rfloor,
 \qquad t\coloneqq\left\lfloor\frac d p\right\rfloor.
\]
Let $e_i$ denote the $i$th standard basis vector in $\Z^n$. Define
\begin{align*}
 H&\coloneqq\sum_{|\beta|\leq t}\eps_{p\beta}y^\beta,\qquad
 G_i\coloneqq\sum_{|\beta|\leq s}\eps_{p\beta+e_i}y^\beta,\quad 1\leq i\leq n,\\
 U&\coloneqq\sum_{\substack{|\alpha|\leq d\\
              \alpha\not\equiv0,e_1,\ldots,e_n\pmod p}}
              \eps_\alpha y^\alpha.
\end{align*}
The congruences are coordinatewise, and the residue classes $0,e_1,\ldots,e_n$ are distinct.
A monomial of total degree at most $d$ in class $0$ has exponent $p\beta$ with
$|\beta|\leq t$. One in class $e_i$ has exponent $p\beta+e_i$ with $|\beta|\leq s$. Thus the
formulas include every monomial in those classes. The $n+2$ polynomials are mutually
independent because they use disjoint sets of the original signs. Since $\eps^p=\eps$ in
$\F_p$, we obtain
\begin{equation}\label{eq:partition}
 g=U+\sum_{i=1}^n y_iG_i^p+H^p.
\end{equation}

This identity uses the original coefficients, with none added or resampled. Differentiating gives
\begin{equation}\label{eq:derivatives}
 \del_i g=\del_i U+G_i^p \qquad(1\leq i\leq n).
\end{equation}
Therefore, each derivative has its own independent block once $U$ has been fixed. The block $H$ does
not occur in any derivative.

\subsection{A degree bound for successive intersections}

We require an affine B\'ezout bound that counts the degrees of all irreducible components,
including components of different dimensions. This is classical (see Heintz \cite{Heintz}
and Fulton \cite[Example~8.4.6]{Fulton}). We give the hypersurface-section argument used
below. The facts about Hilbert polynomials are recalled in
\cite[Chapter~I, \S7]{Hartshorne}.

For a projective variety $T$ of dimension $r$, its degree $\deg T$ is characterized by the
leading term $(\deg T)q^r/r!$ of the Hilbert polynomial of its homogeneous coordinate ring. For
an affine algebraic set $Z$ over an algebraically closed field, define
\[
\delta(Z)\coloneqq\sum_V\deg\overline V,
\]
where $V$ runs over the distinct irreducible components and $\overline V$ denotes the projective
closure of $V$. We set $\delta(\varnothing)\coloneqq0$. In this definition every component is counted
once, without multiplicity. Therefore, the number of components is at most $\delta(Z)$, and the value
of $\delta$ on a finite reduced set is its number of points.

The following lemma bounds the number of components obtained by imposing successive
equations. It also shows that an equation lowers the dimension of a component unless it
vanishes identically there.

\begin{lemma}[Affine B\'ezout inequality]\label{lem:bezout}
Let $D\geq1$ be an integer, and let $h_1,\ldots,h_j$ be polynomials in $n$ variables of degree
at most $D$. Then the following bound holds:
\begin{equation}\label{eq:bezout}
 \delta\bigl(Z(h_1,\ldots,h_j)\bigr)\leq D^j.
\end{equation}
Moreover, if $V$ is irreducible and positive-dimensional and a polynomial $h$ does not vanish
identically on $V$, every irreducible component of $V\cap Z(h)$ has dimension at most
$\dim V-1$.
\end{lemma}
\begin{proof}
We first consider an irreducible affine variety $V$ and a polynomial $h$ of degree at most
$D$. If $h$ vanishes identically on $V$, the intersection is $V$. If $V$ is a point where
$h$ does not vanish, the intersection is empty. A nonzero constant polynomial also gives
an empty intersection. The estimate is immediate in these cases.

We may therefore assume that $V$ is positive-dimensional and that $h$ is nonconstant and
does not vanish identically on $V$. Put $r\coloneqq\dim V$ and $m\coloneqq\deg h$. Let $B$ be the homogeneous coordinate ring of
$\overline V$, and let $\widetilde h$ be the homogenization of $h$ to degree $m$. The ring $B$
is a domain, and the image of $\widetilde h$ is nonzero. Multiplication by that image is
injective. It gives an exact sequence of graded modules
\[
 0\longrightarrow B(-m)\xrightarrow{\ \widetilde h\ }B
 \longrightarrow B/(\widetilde h)\longrightarrow0.
\]
Here the grading shift is defined by $B(-m)_q\coloneqq B_{q-m}$. Write $P_B(q)$ for the
Hilbert polynomial of $B$, and let $T$ be the projective hypersurface section defined on
$\overline V$ by $\widetilde h$. The exact sequence shows that the Hilbert polynomial of $T$ is
$P_B(q)-P_B(q-m)$. Its leading term is
\[
 \frac{m\deg\overline V}{(r-1)!}\,q^{r-1}.
\]
Consequently, $T$ has dimension $r-1$ and scheme-theoretic degree $m\deg\overline V$.

To compare this degree with the degrees of the components, note that every minimal prime
over the nonzero principal ideal $(\widetilde h)$ in the domain $B$ has height one by the
principal ideal theorem. The dimension formula for finitely generated algebras over a field
therefore shows that every irreducible component of the support of $T$ has dimension $r-1$.
Only minimal primes are involved, so we need no assertion about embedded components.

Write $Y_1,\ldots,Y_\ell$ for these irreducible components, and let $T_{\mathrm{red}}$ be the
reduced projective scheme with this support. Passing from $B/(\widetilde h)$ to its reduced
quotient cannot increase the dimension of any homogeneous component.
The Hilbert polynomial of $T_{\mathrm{red}}$ therefore has leading coefficient at most that
of $T$. For completeness, the degree of a reduced equidimensional projective scheme is the sum
of the degrees of its irreducible components. This follows from the exact sequence for a union:
if $J_1,J_2$ are the homogeneous ideals of two closed subsets in a homogeneous coordinate ring
$R$, then
\[
 0\longrightarrow R/(J_1\cap J_2)
 \longrightarrow R/J_1\oplus R/J_2
 \longrightarrow R/(J_1+J_2)\longrightarrow0.
\]
Distinct irreducible components of the same dimension meet in smaller dimension, so the last
term contributes no leading term of that dimension. Induction on the number of components
proves the assertion. Applying it to $T_{\mathrm{red}}$ gives
\[
 \sum_{j=1}^{\ell}\deg Y_j
 \leq m\deg\overline V.
\]

The affine intersection $V\cap Z(h)$ is the part of $T$ in the original affine chart. Its
irreducible components come from the components of $T$ that are not contained in the
hyperplane at infinity. Omitting the components at infinity cannot increase the sum of
degrees. Therefore, we have
\[
 \delta(V\cap Z(h))\leq D\deg\overline V.
\]
The argument also gives the asserted dimension drop when the intersection is nonempty.
The degree bound holds as well in the cases considered at the start, where the component
is retained or the intersection is empty.

For an arbitrary affine algebraic set $Z$, apply the bound to each irreducible component.
Every component of $Z\cap Z(h)$ is among the components of these intersections. Removing
repetitions and nonmaximal components can only decrease the sum of degrees. Thus
\[
 \delta(Z\cap Z(h))\leq D\delta(Z).
\]
Starting with affine space, whose degree is $1$, and applying this inequality successively to
$h_1,\ldots,h_j$ proves \eqref{eq:bezout}.
\end{proof}

\subsection{The critical locus}
We now impose the derivative equations one at a time. The degree bound limits the number
of components to consider, and the independent-column estimate bounds the probability that
a derivative vanishes identically on a positive-dimensional component.

Put $D\coloneqq d-1$ and
\[
 W_i\coloneqq Z(\del_1g,\ldots,\del_i g)\subset\A^n_{\overline{\F}_p},
 \qquad W_0\coloneqq\A^n.
\]

The following proposition bounds the probability that the critical locus is
positive-dimensional. It also bounds the size of a finite critical locus, which is
independent of the unused block $H$.

\begin{proposition}\label{prop:critical}
With the preceding notation, the probability of a positive-dimensional critical locus satisfies
\begin{equation}\label{eq:critical}
 \Prob(\dim W_n>0)
 \leq\left(\sum_{i=0}^{n-1}D^i\right)2^{-(s+1)}
 \leq n d^{n-1}2^{-(s+1)}.
\end{equation}
When $W_n$ is finite, it has at most $D^n$ geometric points. It is determined by
$U,G_1,\ldots,G_n$ and is independent of $H$.
\end{proposition}
\begin{proof}
Condition on $U$, and reveal $G_1,\ldots,G_n$ in order. Consider the first index $i$ at which
the inequality
\[
 \dim W_i\leq n-i
\]
fails. Before this index, $W_{i-1}$ has dimension at most $n-i+1$. A failure at $i$ can therefore occur only if
an irreducible component $V$ of $W_{i-1}$ of dimension $n-i+1$ is contained in $Z(\del_i
U+G_i^p)$. Indeed, other components already have dimension at most $n-i$, and a proper
hypersurface intersection lowers the dimension of each remaining component. There are at most
$D^{i-1}$ possible components by Lemma~\ref{lem:bezout}.

Every such $V$ is positive-dimensional. At least one coordinate function $y_j$ is nonconstant on
$V$. Otherwise every point of $V$ would have the same coordinates, contrary to its positive
dimension. Let
$\overline{\F}_p(V)$ denote the function field of $V$. In this field, the chosen coordinate
function is transcendental over $\overline{\F}_p$: an element algebraic over an algebraically
closed field belongs to that field. The powers
\[
 1,y_j,\ldots,y_j^s
\]
are therefore linearly independent over $\overline{\F}_p$, and hence over $\F_p$.

The condition $\del_i U+G_i^p=0$ on $V$ either is impossible or prescribes a single value for
the restriction of $G_i$ to $\overline{\F}_p(V)$. Indeed, two admissible restrictions would have
the same $p$th power and hence be equal. This uses injectivity of Frobenius in a field, not its
surjectivity. The displayed powers are distinct monomials in $G_i$. Lemma~\ref{lem:rank} shows
that the conditional probability of this event is at most $2^{-(s+1)}$.

The component $V$ depends on $U,G_1,\ldots,G_{i-1}$ and not on $G_i$, so the signs used in
this estimate remain unconditioned. Taking a union bound over the components and the first
failing index proves \eqref{eq:critical}. If no failure occurs, $W_n$ is empty or
zero-dimensional. In the latter case, Lemma~\ref{lem:bezout} bounds its number of points by
$D^n$. Equation~\eqref{eq:derivatives} also shows that $H$ is not used in the definition of
$W_n$.
\end{proof}

\subsection{The remaining equation at a high-degree point}
When the critical locus is finite, we use the remaining block $H$ to determine which of its
points lie on the hypersurface. A point of large residue degree gives many independent
evaluation directions. The following bound includes both the probability that the critical
locus is not finite and the probability that a high-degree point satisfies the remaining
equation.

\begin{proposition}\label{prop:high}
For every integer $R\geq0$, the following bound holds:
\begin{align}\label{eq:high}
 &\Prob(\exists\text{ singular closed }P\in\A^n_{\F_p}:\deg P>R)\notag\\
 &\qquad\leq\left(\sum_{i=0}^{n-1}D^i\right)2^{-(s+1)}
             +D^n2^{-\min(t+1,R+1)}.
\end{align}
\end{proposition}
\begin{proof}
Exclude the event $\dim W_n>0$, whose probability is bounded by Proposition~\ref{prop:critical}.
Condition on $U,G_1,\ldots,G_n$. Then $W_n$ is a fixed finite set invariant under Frobenius. Its
closed points are the Frobenius orbits of its geometric points, so it has at most $D^n$ closed
points. The equation $g(P)=0$ at a closed point $P\in W_n$ is
\[
 H(P)^p=-U(P)-\sum_{i=1}^n y_i(P)G_i(P)^p.
\]
The residue field $\kappa(P)$ is finite, so its Frobenius map is bijective. This equation
prescribes one value of $H(P)$.

Write $e\coloneqq\deg P$. The coordinate images together generate $\kappa(P)$ as an
$\F_p$-algebra, although no single coordinate need generate the field.
Lemma~\ref{lem:filtration} shows that evaluation of polynomials of degree at most $t$ has
rank at least $\min(t+1,e)$. The coefficients of $H$ remain independent uniform signs.
Lemma~\ref{lem:rank} therefore bounds the probability of the prescribed value by
\[
 2^{-\min(t+1,e)}\leq2^{-\min(t+1,R+1)}\quad(e>R).
\]
Sum this bound over the at most $D^n$ candidate points and restore the excluded event.
Independence between evaluations at different points is not required.
\end{proof}

\section{A finite-field estimate}\label{sec:finite}

We now combine the estimates for the two degree ranges. We leave the prime and cutoff
unspecified, since the qualitative and quantitative results use different choices.

Define
\[
 \sigma_{n,d}(p)\coloneqq\Prob(X_{\bar f}\text{ is singular over }\overline{\F}_p).
\]
This probability refers to the original sign distribution reduced modulo $p$, not the uniform
distribution on all forms over $\F_p$.

\Needspace{5\baselineskip}
The first sum below bounds the contribution from points of small degree. The other two terms bound the
probability that the dimension-reduction argument fails and the probability that a
high-degree critical point lies on the hypersurface.

\begin{theorem}\label{thm:finite}
Let $n\geq1$ and $d\geq3$ be integers, and let $p$ be an odd prime. Put
\[
 D\coloneqq d-1,\qquad s\coloneqq\left\lfloor\frac{d-1}{p}\right\rfloor,
 \qquad t\coloneqq\left\lfloor\frac d p\right\rfloor.
\]
Let $R$ be an integer satisfying $0\leq R\leq\lfloor(d+1)/(n+1)\rfloor$, and define $K_e$ by
\eqref{eq:Ke}. Then the singularity probability satisfies
\begin{align}\label{eq:finite}
 \sigma_{n,d}(p)
 &\leq\frac{p}{p-1}\sum_{e=1}^R
             \frac{p^{ne}}e\Th_{K_e}(p)^{(n+1)e}\notag\\
 &\quad +(n+1)\left\{
      \left(\sum_{i=0}^{n-1}D^i\right)2^{-(s+1)}
             +D^n2^{-\min(t+1,R+1)}\right\}.
\end{align}
\end{theorem}
\begin{proof}
Equation~\eqref{eq:low} accounts for closed points of degree at most $R$ in projective space.
Apply Proposition~\ref{prop:high} on each of the $n+1$ standard affine charts for the remaining
degrees. Every geometric point is defined over a finite extension of $\F_p$, so these two degree
ranges cover all potential singular points. A union bound gives \eqref{eq:finite}.
\end{proof}

For the qualitative conjecture, it is enough to use $p^2$ bases at each point. This gives
the following simpler estimate. For a fixed prime, the degree-dependent error tends to
zero exponentially, leaving a term of order $1/p$.

\begin{corollary}\label{cor:simple}
Let $p$ be a prime at least $2^{n+2}$, and set
\[
 R_0\coloneqq\left\lfloor\frac{d+1}{(n+1)p^2}\right\rfloor.
\]
Then the following estimate holds:
\begin{equation}\label{eq:simple}
 \sigma_{n,d}(p)\leq\frac{3\cdot2^{n+1}}p
                  +(n+1)^2d^n2^{-(R_0+1)}.
\end{equation}
\end{corollary}
\begin{proof}
For a point of degree $e$ with $1\leq e\leq R_0$, the definition of $K_e$ gives $K_e\geq p^2$.
Since $\Th_K(p)$ is nonincreasing in $K$, Lemma~\ref{lem:theta} gives
\[
 \Th_{K_e}(p)\leq\Th_{p^2}(p)<\frac2p.
\]
Proposition~\ref{prop:local} therefore bounds the singularity probability at this point by
$(2/p)^{(n+1)e}$. Put $a\coloneqq2^{n+1}/p$, which is at most $1/2$. By \eqref{eq:point-count}, the
low-degree sum is at most
\[
 \frac32\sum_{e\geq1}\frac{a^e}{e}
 \leq\frac32\frac{a}{1-a}\leq3a.
\]
The inequality $d+1\leq(n+1)p(d-1)$ implies $R_0\leq s\leq t$ for the parameters in question. The
contribution of the high-degree points in Theorem~\ref{thm:finite} is at most
\[
 (n+1)(n d^{n-1}+d^n)2^{-(R_0+1)}
 \leq(n+1)^2d^n2^{-(R_0+1)}.
\]
This proves the assertion, including the case $R_0=0$.
\end{proof}

\section{Proofs of the main results}\label{sec:proofs}

We apply the finite-field estimates to prove the results over the complex numbers. For the
quantitative smoothness bound, the prime grows with the degree. A fixed small prime gives
the exponential bound for positive-dimensional singular loci and the resulting
irreducibility estimate.

\subsection{Proof of the quantitative bound}
We choose the prime comparable to $\sqrt d$ and the degree cutoff comparable to $\log d$.
Points of degree one contribute $O_n(d^{-1/2})$, and the other terms in the sieve are
smaller.

\begin{proof}[Proof of Theorem~\ref{thm:main}]
Fix $n$, and first suppose $d$ is sufficiently large in terms of $n$. Bertrand's postulate (see
the elementary proof in \cite{Erdos}), applied to $\lfloor\sqrt d\rfloor$, gives an odd prime
satisfying
\begin{equation}\label{eq:prime-choice}
 \tfrac12\sqrt d\leq p\leq2\sqrt d.
\end{equation}
Choose
\[
 R\coloneqq\left\lceil(n+2)\log_2 d\right\rceil.
\]
Recall the definitions $D\coloneqq d-1$, $s\coloneqq\lfloor(d-1)/p\rfloor$, and
$t\coloneqq\lfloor d/p\rfloor$ from Theorem~\ref{thm:finite}. The cutoff $R$ grows logarithmically with $d$, whereas $s$ grows on
the order of $\sqrt d$. Thus, for sufficiently large $d$ in terms of $n$, we have
\[
 R\leq s\leq t,\qquad R\leq\left\lfloor\frac{d+1}{n+1}\right\rfloor.
\]
In particular, Theorem~\ref{thm:finite} applies.

For every integer $e$ between $1$ and $R$, the number $(d+1)/((n+1)e)$ is at least $1$.
Using the inequality $\lfloor x\rfloor\geq x/2$, valid for every real $x\geq1$, we obtain
\[
 K_e\geq\frac{d+1}{2(n+1)e}.
\]
Put $A_n\coloneqq2+\sqrt{2(n+1)}$. Lemma~\ref{lem:theta} and
\eqref{eq:prime-choice} imply
\begin{equation}\label{eq:theta-choice}
 \Th_{K_e}(p)
 \leq\frac2{\sqrt d}+\sqrt{\frac{2(n+1)e}{d+1}}
 \leq A_n\sqrt{\frac e d}.
\end{equation}
Put $a\coloneqq(n+1)/2$ and $B_n\coloneqq2^n A_n^{n+1}$. Since $p^n\leq2^n d^{n/2}$, the low-degree contribution
in \eqref{eq:finite} is at most
\begin{equation}\label{eq:optimized-low}
 \frac32\sum_{e=1}^R\frac1e
          \left(\frac{B_n e^a}{\sqrt d}\right)^e.
\end{equation}
The term with $e=1$ is $3B_n/(2\sqrt d)$. Let
\[
 \rho\coloneqq\frac{B_n R^a}{\sqrt d}.
\]
We have $\rho\leq1/2$ for all sufficiently large $d$. The remaining terms in
\eqref{eq:optimized-low} are at most
\[
 \frac32\sum_{e=2}^R\rho^e
 \leq\frac32\frac{\rho^2}{1-\rho}
 \ll_n\frac{(\log d)^{n+1}}d
 =o_n(d^{-1/2}).
\]
This bounds the whole low-degree contribution by $O_n(d^{-1/2})$.

For the high-degree contribution, \eqref{eq:prime-choice} gives
\[
 s+1>\frac{d-1}{p}\geq\frac{\sqrt d}{3}\qquad(d\geq3).
\]
Consequently, we have
\begin{align*}
 (n+1)\left(\sum_{i=0}^{n-1}D^i\right)2^{-(s+1)}
 &\leq n(n+1)d^{n-1}2^{-\sqrt d/3},\\
 (n+1)D^n2^{-\min(t+1,R+1)}
 &\leq(n+1)d^n2^{-(R+1)}
 \leq\tfrac12(n+1)d^{-2}.
\end{align*}
The first expression is $o_n(d^{-2})$. Combining these bounds with \eqref{eq:finite} and
\eqref{eq:prob-specialization} proves \eqref{eq:main} for large $d$. Increasing $C_n$ handles
the finitely many remaining degrees.
\end{proof}

There is also a simpler proof of the qualitative conjecture, in which the degree and the
prime tend to infinity successively.

\begin{remark}
Fix a prime $p\geq2^{n+2}$ and let $d$ tend to infinity in
Corollary~\ref{cor:simple}. By \eqref{eq:prob-specialization}, we obtain
\[
 \limsup_{d\to\infty}\Prob(X_f\text{ is singular})
 \leq\frac{3\cdot2^{n+1}}p.
\]
Letting $p$ tend to infinity proves that the limit is zero. The proof of Theorem~\ref{thm:main}
improves the quantitative estimate by allowing the number of bases to depend on the degree of
the closed point. It uses no independence between reductions at different primes.
\end{remark}

\subsection{Positive-dimensional singularities and irreducibility}\label{sec:positive}
A positive-dimensional singular locus gives a positive-dimensional critical locus on some
affine chart. Thus the dimension-reduction estimate suffices, and we do not need the local
concentration bound. We apply the estimate at a fixed prime and then use the intersection
of two hypersurface factors to deduce the irreducibility bound.

\begin{proof}[Proof of Theorem~\ref{thm:positive}]
First work over $\overline{\F}_p$ for an odd prime $p$. If the projective singular locus has a
positive-dimensional irreducible component, some standard affine chart meets that component in a
nonempty open subset. The singular locus in that chart is positive-dimensional and is contained
in the critical locus $W_n$. Proposition~\ref{prop:critical} and a union bound over the $n+1$
charts give
\begin{equation}\label{eq:positive-finite}
 \Prob(\dim\Sing X_{\bar f}\geq1)
 \leq n(n+1)d^{n-1}2^{-\lfloor(d-1)/p\rfloor-1}.
\end{equation}
Apply Lemma~\ref{lem:specialization} and take $p=3$ to obtain \eqref{eq:positive} over $\C$.
This argument does not use the low-degree concentration estimate.
\end{proof}

To deduce irreducibility, we use the following standard intersection argument, which also
covers repeated factors (compare the remark following \cite[Theorem~3.2]{Poo}).

\begin{proof}[Proof of Corollary~\ref{cor:absolute}]
Suppose that the ambient dimension is at least three and that $f$ is reducible. Write $f=AB$
with nonconstant polynomials over $\C$.
Both factors are homogeneous. Indeed, in a polynomial ring over a field, the smallest and
largest total degrees occurring in a product are the sums of the corresponding degrees in its
factors. Since these degrees coincide for $f$, they coincide for each factor.

Let $Y\subset\PP^n_\C$ be the projective set defined by $A=B=0$. The ideal $(A,B)$ in
$\C[x_0,\ldots,x_n]$ has height at most $2$ by the principal ideal theorem. Its affine zero set
therefore has dimension at least $n-1$. This set is a homogeneous cone, so its projectivization
$Y$ has dimension at least $n-2$. At every point of $Y$, the equation $f$ and all its first
derivatives vanish by the product rule. Thus
\[
 Y\subseteq\Sing X_f,
 \qquad \dim\Sing X_f\geq n-2\geq1.
\]
This includes the case where $A$ and $B$ have a common factor. Failure of absolute
irreducibility therefore implies the event already bounded in \eqref{eq:positive}.
\end{proof}

\section{AxiomProver Lean certificate}\label{sec:AxiomProver}

This paper was written using human-AI collaboration. Working with AxiomProver, the authors began from the statement of the conjecture and proceeded through literature search, numerical experimentation, and the proposal and rejection of many candidate strategies before the proofs below took shape. At every stage the authors set the direction and made the mathematical judgments, and they take full responsibility for the correctness of the results. AxiomProver is an AI system for mathematical research through formal proof, currently under development at Axiom Math.

Theorems~\ref{thm:main}
and~\ref{thm:positive}, together with Corollaries~\ref{cor:BS} and~\ref{cor:absolute},
have been formalized in Lean with AxiomProver, assuming the existing literature specified
below. The repository contains the formal statements, proofs, and build
instructions:

\begin{center}
\url{https://github.com/AxiomMath/BrowningSawin}
\end{center}

Beyond Lean and Mathlib, the development assumes five mathematical statements
from the literature as hypotheses to the main targets. Their mathematical content are as follows.
\begin{enumerate}
\item The affine B\'ezout inequality. For $j$
polynomials in $n$ variables over an algebraically closed field, each of degree
at most $D\geq1$, the sum of the degrees of the projective closures of all
irreducible components of their common zero set is at most $D^j$ (for example, see
\cite{Heintz} or \cite[Example~8.4.6]{Fulton}).

\item Dimension drop under a nonvanishing equation. If $V$ is an irreducible
positive-dimensional affine algebraic set over an algebraically closed field
and $h$ does not vanish identically on $V$, every irreducible component of
$V\cap Z(h)$ has dimension at most $\dim V-1$ (for example, see 
\cite[Chapter~I, \S7]{Hartshorne} or \cite[Example~8.4.6]{Fulton}).

\item Specialization of projective dimension. For a homogeneous ideal with integer
coefficients, the projective zero scheme over $\overline{\F}_p$ obtained by
reduction modulo any prime $p$ has dimension at least that of the corresponding
scheme over $\C$. This is the projective specialization case of upper
semicontinuity of fiber dimension for proper morphisms
\cite[Tag~0D4I]{Stacks}.

\Needspace{6\baselineskip}
\item The affine hypersurface Jacobian criterion. At a geometric point on a
hypersurface defined by a nonzero polynomial over a perfect field, the point
lies outside the smooth locus if and only if all first partial derivatives
vanish there. The hypersurface retains its scheme structure, including
repeated factors \cite[Chapter~I, \S5; Chapter~II, \S8]{Hartshorne}.

\item Hilbert--Serre eventual polynomiality. The projective Hilbert function used in
the development agrees with a polynomial over $\Q$ in all sufficiently large
degrees. This supplies the existence of the Hilbert polynomial used to define
projective degree \cite[Chapter~I, \S7]{Hartshorne}.
\end{enumerate}


\begin{thebibliography}{99}

\bibitem{BSK}
L.~Bary-Soroker and G.~Kozma,
\emph{Is a bivariate polynomial with $\pm1$ coefficients irreducible? Very likely!},
International Journal of Number Theory \textbf{13} (2017), no.~4, 933--936.
\href{https://doi.org/10.1142/S1793042117500488}{doi:10.1142/S1793042117500488}.
\url{https://arxiv.org/abs/1602.06530}.

\bibitem{BS}
T.~Browning and W.~Sawin,
\emph{Random Diophantine equations of large degree},
preprint (2025), arXiv:2510.26191.
\url{https://arxiv.org/abs/2510.26191}.

\bibitem{Erdos}
P.~Erd\H{o}s,
\emph{Beweis eines Satzes von Tschebyschef},
Acta Litterarum ac Scientiarum Regiae Universitatis Hungaricae Francisco-Josephinae,
Sectio Scientiarum Mathematicarum \textbf{5} (1932), 194--198.
\url{https://acta.bibl.u-szeged.hu/13396/}.

\bibitem{Fulton}
W.~Fulton,
\emph{Intersection Theory}, second edition,
Ergebnisse der Mathematik und ihrer Grenzgebiete, 3.~Folge, vol.~2,
Springer-Verlag, Berlin, 1998.
\href{https://doi.org/10.1007/978-1-4612-1700-8}{doi:10.1007/978-1-4612-1700-8}.

\bibitem{Hartshorne}
R.~Hartshorne,
\emph{Algebraic Geometry}, Graduate Texts in Mathematics, vol.~52,
Springer-Verlag, New York--Heidelberg, 1977.
\href{https://doi.org/10.1007/978-1-4757-3849-0}{doi:10.1007/978-1-4757-3849-0}.

\bibitem{Heintz}
J.~Heintz,
\emph{Definability and fast quantifier elimination in algebraically closed fields},
Theoretical Computer Science \textbf{24} (1983), no.~3, 239--277.
\href{https://doi.org/10.1016/0304-3975(83)90002-6}{doi:10.1016/0304-3975(83)90002-6}.

\bibitem{HPX}
J.~He, H.~T.~Pham, and M.~W.~Xu,
\emph{Universality for low-degree factors of random polynomials over finite fields},
International Mathematics Research Notices \textbf{2023} (2023), no.~17, 14752--14794.
\href{https://doi.org/10.1093/imrn/rnac239}{doi:10.1093/imrn/rnac239}.
Preprint version: \url{https://arxiv.org/abs/2201.06156v2}.
References to Proposition~3.3 use the numbering of this version.

\bibitem{KZ}
G.~Kozma and O.~Zeitouni,
\emph{On common roots of random Bernoulli polynomials},
International Mathematics Research Notices \textbf{2013} (2013), no.~18, 4334--4347.
\href{https://doi.org/10.1093/imrn/rns164}{doi:10.1093/imrn/rns164}.
\url{https://arxiv.org/abs/1109.2316}.

\bibitem{PSZ}
R.~Peled, A.~Sen, and O.~Zeitouni,
\emph{Double roots of random Littlewood polynomials},
Israel Journal of Mathematics \textbf{213} (2016), no.~1, 55--77.
\href{https://doi.org/10.1007/s11856-016-1328-3}{doi:10.1007/s11856-016-1328-3}.
\url{https://arxiv.org/abs/1409.2034}.

\bibitem{Poo}
B.~Poonen,
\emph{Bertini theorems over finite fields},
Annals of Mathematics (2) \textbf{160} (2004), no.~3, 1099--1127.
\href{https://doi.org/10.4007/annals.2004.160.1099}{doi:10.4007/annals.2004.160.1099}.
\url{https://math.mit.edu/~poonen/papers/bertini.pdf}.

\bibitem{Stacks}
\emph{The Stacks Project}, Lemma~37.30.5, Tag~0D4I.
\url{https://stacks.math.columbia.edu/tag/0D4I}.

\end{thebibliography}
\end{document}